\documentclass[12pt]{amsart}
\usepackage{cases}
\usepackage{amscd}
\usepackage{amsmath}
\usepackage{amsthm}
\usepackage{amssymb}
\usepackage{mathrsfs}
\usepackage{amsfonts}
\usepackage{color}

\usepackage{pifont}

\usepackage{upgreek}
\usepackage{bm}

\usepackage{indentfirst, latexsym, bm, amsmath, eufrak, amsthm}
\usepackage{amsmath}
\usepackage{amsthm}
\usepackage{amssymb}
\usepackage{mathrsfs}
\usepackage{amsfonts}

\usepackage{pifont}

\usepackage{upgreek}
\usepackage{bm}
\newcommand*{\tr}{\mathrm{tr}}
\numberwithin{equation}{section}
\newtheorem{theo}{Theorem} 

\newtheorem{lem}{Lemma}

\newtheorem{remark}{Remark}
\newtheorem{prop}{Proposition}
\newtheorem{definition}{Definition}

\begin{document}
\title{Totally Real Flat Minimal Surface in Hyperquadric}
\author[Ling He]{Ling He}
\author[Xiaoxiang Jiao]{Xiaoxiang Jiao}
\author[Mingyan Li]{Mingyan Li$^*$}

\address{Ling He:  Center for Applied Mathematics, Tianjin University, Tianjin
	300072, P. R. China}

\email{heling@tju.edu.cn}

\address{Xiaoxiang Jiao:  School of Mathematical Sciences, University of Chinese Academy of Sciences, Beijing
	100049, P. R. China}

\email{xxjiao@ucas.ac.cn}

\address{Mingyan Li (Corresponding author): School of Mathematics and Statistics, Zhengzhou University, Zhengzhou 450001, P. R. China}

\email{limyan@zzu.edu.cn}

\thanks{This work was supported by NSFC (Grant Nos. 11901534, 11871450, 11501548) and CSC Visiting Scholar Program}

\maketitle

\begin{abstract}
In this paper, we study geometry of totally real minimal  surfaces in the complex hyperquadric $Q_{N-2}$, and obtain some characterizations of the harmonic sequence generated by these minimal immersions. For totally real flat surfaces that are minimal in both $Q_{N-2}$ and $\mathbb{C}P^{N-1}$, we determine them for $N=4, 5, 6$, and give a classification theorem when they are Clifford solutions.
\end{abstract}

  \section{Introduction}

A beautiful and quite complete theory has been developed for the construction of harmonic maps (equivalently, minimal branched immersions) from the two-sphere $S^2$ to complex Grassmann manifold $G(k,N;\mathbb{C})$.  In \cite{[5], [30]}, the authors gave unique factorization theorems and obtained an explicit bijective parametrization of all harmonic maps of $S^2$ in $G(k,N;\mathbb{C})$ by cetain finite sequences of holomorphic maps from $S^2$ to complex Grassmannians together with an algorithm involving only algebraic operations and Cauchy integrations for finding the harmonic map corresponding to a given sequence of holomorphic maps. The same problem was also studied by Chern and Wolfson, using their \emph{crossing and recrossing} construction(\cite{[6], [31]}). When $k=1$, this reduces to the well-known parametrization of harmonic maps from $S^2$ to a complex projective space by a single holomorphic map. In 1986, Wolfson \cite{[14]} studied harmonic map $S^2 \rightarrow G(2,N;\mathbb{R})$, or the complex hyperquadric $Q_{N-2}$, he reduced the problem to finding a certain subset $S_0$ of the set $S$ of degenerate harmonic maps $S^2 \rightarrow G(2,N;\mathbb{R})$. Three years later, by certain flag transforms called \emph{forward and backward replacement}, 
Bahy-El-Dien and Wood \cite{[2]}  proved that  these harmonic maps can all  be obtained in a unique way from a holomorphic map.  With this result and the rigidity theorem of minimal immersion $S^2 \rightarrow  \mathbb{C}P^{N-1}$ with constant curvature in \cite{[3]}, a series of classification results about minimal two-spheres of constant Gauss curvature in $Q_{N-2}$ are obtained under certain conditions (cf. \cite{[7], [51], [60], [54], [16]}).

It is desirable to study harmonic maps from Riemann surfaces of higher genus. When $k=1$, i.e., $G(1,N;\mathbb{C})=\mathbb{C}P^{N-1}$, there is a family of totally real isometric harmonic maps from the flat complex plane $\mathbb{C}$ into $\mathbb{C}P^{N-1}$. None of these maps is pseudoholomorphic. This family was first described by Kenmotsu in \cite{[20]} and, then Bolton and Woodward in \cite{[18]}. In 1995, Jensen and Liao \cite{[19]} discovered continuous families of noncongruent flat minimal tori in $\mathbb{C}P^{N-1}$ by this family of minimal surfaces.

In this paper we generalize the result of \cite{[18], [20]}  to totally real minimal surfaces in $Q_{N-2}\subseteq \mathbb{C}P^{N-1}$ by theory of harmonic maps. Our strategy is to establish useful criteria for deciding whether or not a totally real flat harmonic map from $\mathbb{C}$ to $\mathbb{C}P^{n} (n \leq N-1)$ can be transformed into $Q_{N-2}$ via some unitary matrix in $U(N)$, where we use the standard method of matrix analysis, which is effective in such problems, and the harmonic sequence of a harmonic map  also play an important role.

Our paper is organized as follows.    In Section 3, as our starting point, we investigate totally real flat minimal surfaces in $Q_{N-2}$ by regarding them as harmonic maps in $G(2,N;\mathbb{R})$, and get some properties with respect to harmonic sequences. In Section 4, we analyze solutions of totally real flat immersion $f:\mathbb{C} \rightarrow Q_{N-2}\subseteq \mathbb{C}P^{N-1}$, which is minimal in both  $Q_{N-2}$ and $\mathbb{C}P^{N-1}$. Furthermore we give separate treatments of the cases $Q_{2}, Q_{3}$ and $Q_{4}$ because we find their understanding is basic for that of the general case. Finally in Section 5, moving on, we study and classify the interesting and important case that $f:\mathbb{C} \rightarrow Q_{N-2}\subseteq \mathbb{C}P^{N-1}$ is a Clifford solution, which is better understood than the general case.
    	
  \section{Geometry of surfaces in hyperquadric}

  For any $N=1, 2, ...$, let $\langle \cdot, \cdot\rangle$ denotes the standard Hermitian inner product on $\mathbb{C}^N$, which is defined by 
  $$\langle z, w\rangle=z_1\overline{w}_1+z_2\overline{w}_2+...+z_N\overline{w}_N,$$ where $z=(z_1, z_2,..., z_N)^{T}, w=(w_1, w_2,..., w_N)^{T} \in \mathbb{C}^N$ and $\overline{w}_i$ is the complex conjugation of $w_i, \ i=1,2, ..., N$.
  
  Complex hyperquadric space $Q_{N-2}:=\{[X]\in \mathbb{C}P^{N-1} |  \langle \overline{X},  X\rangle=0\}$  is a complex algebraic submanifold in  complex projective space $\mathbb{C}P^{N-1}$. 
  Considering  $G(2,N;\mathbb{C})$ as
  the set of Hermitian orthogonal projection from $\mathbb{C}^N$ onto
  a $2$-dimensional subspace in $\mathbb{C}^N$ with standard Riemannian matric, its fixed point set under complex conjugation is $G(2,N;\mathbb{R}):=\{\phi \in G(2,N;\mathbb{C})| \overline{\phi}=\phi\}$, which lies totally geodesically  in
$G(2,N;\mathbb{C})$. Then $Q_{N-2}$ and $G(2,N;\mathbb{R})$ can be identified by 
 $$Q_{N-2} \rightarrow G(2,N;\mathbb{R}),   \quad  \  [X] \mapsto
\frac{\sqrt{-1}}{2} \overline{\mathrm{X}} \oplus  \mathrm{X},$$ where $[X]\in Q_{N-2}$
and $\mathrm{X}$ is a homogeneous coordinate vector of $[X]$. It is clear that
the map  is one-to-one and onto, and it is an isometry (for more details see \cite{[14]}). Therefore in our paper, we treat $Q_{N-2}$ and $G(2,N;\mathbb{R})$ in the same way, and $Q_{N-2}$ carries the standard Riemannian metric of  $G(2,N;\mathbb{R})$.

 Let $M$ be a connected Riemannian surface with complex coordinate $(z, \overline{z})$.  
  Suppose  \begin{equation}\phi:M\rightarrow G(2,N;\mathbb{C})\label{eq:2.1}\end{equation}
   is a minimal immersion, which is equivalent to say that the image of $\phi$ is a conformal harmonic surface in $G(2,N;\mathbb{C})$.  We will also assume that all maps \eqref{eq:2.1} are \emph{linearly full} in the sense that the image is not contained in any hyperplane of $G(2,N;\mathbb{C})$.

 Treat  $\phi:
   M\rightarrow G(2,N;\mathbb{C})$  as a Hermitian orthogonal projection
   onto a $2$-dimensional subbundle $\underline{\phi}$ of the
   trivial bundle $\underline{\mathbb{C}}^N = M \times \mathbb{C}^N$
   given by setting the fibre $\underline{\phi}_x = \phi (x)$ for
   all $x\in M$.  Then $\underline{\phi}$ is called (a)
   \emph{harmonic ((sub-) bundle) } whenever $\phi$ is a harmonic
   map.  From  \cite{[5], [6]}, two harmonic sequences are derived as follows:
 $$\underline{\phi}=\underline{\phi}_{0}\stackrel{\partial{'}}
 	{\longrightarrow }\underline{\phi}_1\stackrel{\partial{'}}
 	{\longrightarrow}\cdots\stackrel{\partial{'}}
 	{\longrightarrow}\underline{\phi}_ {i}\stackrel{\partial{'}}
 	{\longrightarrow}\cdots,$$
$$\underline{\phi}=\underline{\phi}_{0}\stackrel{\partial{''}}
 	{\longrightarrow}\underline{\phi}_{-1}\stackrel{\partial{''}}
 	{\longrightarrow}\cdots\stackrel{\partial{''}} {\longrightarrow}
 	\underline{\phi}_{-i}\stackrel{\partial{''}}{\longrightarrow}\cdots,
 	$$ where $\underline{\phi}_{i} = \partial ^\prime
 \underline{\phi} _{i-1}$ and $\underline{\phi}_{-i}=
 \partial ^{ \prime \prime} \underline{\phi} _{-i+1} $ are Hermitian
 orthogonal projections from $M \times \mathbb{C} ^N$ onto
 ${\underline{Im}}\left(\phi^{\perp}_ {i-1}\partial \phi_{i-1}
 \right)$ and ${\underline{Im}}\left(\phi^{\perp}_
 {-i+1}\overline{\partial} \phi_ {-i+1}\right)$ respectively, $i=1,2,\ldots$.
 This is well defined as long as $\phi_ {i-1}$ (resp. $\phi_ {-i+1}$) is not identically zero, in which case the process stops and anti-holomorphic   (resp. holomorphic) curve occurs. In this paper we exclude this case, which means that the harmonic sequence will always extend infinitely in both directions.

   Denote
   \begin{equation}\partial=\frac{\partial}{\partial z}, \ \overline{\partial}=\frac{\partial}{\partial \overline{z}}, \ A_z=(2\phi-I)\partial\phi, \ A_{\overline{z}}=(2\phi-I)\overline{\partial}\phi.\label{eq:2.14}\end{equation} The metric induced by $\phi$ on $M$ is locally given by
   $$ds^2=-\tr A_zA_{\overline{z}}dzd\overline{z},$$ 
   and its second fundamental form $B$ is as follows 
   $$\|B\|^2 = 4\textrm{tr}PP^{\ast},$$
   where $P=\partial\left(\frac{A_z}{4\lambda^2}\right), \
   P^{\ast} =- \overline{\partial}
   \left(\frac{A_{\overline{z}}}{4\lambda^2}\right)$ with
   $4\lambda^2=-\tr A_zA_{\overline{z}}$ (cf. \cite{[30]}).
  Set 
    $$s:\phi-{\phi}^{\bot}=2\phi-I,$$ which is a map from $M$ into the unitary group $U(N)$. It is well-known that $\phi$ is harmonic if and only if $s$ is harmonic.
   To analyze $\phi$, we need one of Uhlenbeck's results (cf. \cite{[17]}) as follows:
   \begin{lem}
   	Let $s: M\rightarrow U(N)$ be a smooth map. Then $s$ is a harmonic map if and only if it satisfies the following equation
   	\begin{equation}\overline{\partial}A_z=[A_z, A_{\overline{z}}], \label{eq:2.2}\end{equation}
   	where $A_z=\frac{1}{2}s^{-1}\partial s, A_{\overline{z}}=\frac{1}{2}s^{-1}\overline{\partial} s$.
   \end{lem}

\begin{remark}
In Lemma 1, when $s=2\phi-I$,  we have $\frac{1}{2}s^{-1}\partial s=(2\phi-I)\partial\phi$ and $\frac{1}{2}s^{-1}\overline{\partial} s=(2\phi-I)\overline{\partial}\phi$. Thus in the absence of confusion, we still denote them by $A_z$ and $A_{\overline{z}}$ respectively.
\end{remark}

  Given any minimal immersion  $\psi: M \rightarrow Q_{N-2}$, suppose  its corresponding map  of $M$ in $G(2,N;\mathbb{R})$  is $\phi$. Then there exists unordered pair of local sections of $\phi$, $\{\overline{X}, X\}$ such that $\psi=[X]$ and
  $\phi=\overline{X}\oplus X$. 
Without loss of generality, we write $|X|=1$.  Define 
  \begin{equation} 
  \xi:=\partial X- \langle \partial X, X\rangle X, \quad \eta:=\overline{\partial} X- \langle \overline{\partial} X, X\rangle X.\label{eq:2.3}
  \end{equation}
  The metric condition gives
  \begin{equation}\langle \xi, \eta\rangle=\langle \xi, X\rangle=\langle \eta, X\rangle=0,\label{eq:2.4}
  \end{equation}
  and \begin{equation} |\xi|^2+|\eta|^2=2\lambda^2.\label{eq:2.5}
  \end{equation}
 Notice that $\phi$ can be written in the form $\phi=\overline{X}(\overline{X})^*+XX^*$, by \eqref{eq:2.14} and \eqref{eq:2.3} we derive
\begin{equation}
A_z=X\eta^*+\overline{X}(\overline{\xi})^*-\xi X^*-\overline{\eta}(\overline{X})^*, A_{\overline{z}}=X\xi^*+\overline{X}(\overline{\eta})^*-\eta X^*-\overline{\xi}(\overline{X})^*.\label{eq:3.31}\end{equation} 
Combining it with Lemma 1, we have

\begin{prop}
	A smooth map $\phi:M\rightarrow G(2,N;\mathbb{R})$ is harmonic if and only if 
	\begin{equation}\overline{\partial} \xi =-|\xi|^2X-\langle \xi, \overline{\eta}\rangle\overline{X}+\langle \overline{\partial}X, X\rangle\xi, \quad \partial\eta=-|\eta|^2X-\langle \xi, \overline{\eta}\rangle\overline{X}+\langle \partial X, X\rangle\xi \label{eq:3.4}\end{equation}
	hold.
\end{prop}

Using it, we immediately get
\begin{prop}
	Let $\psi: M\rightarrow Q_{N-2}$ be a minimal immersion. Then $\psi$ is minimal in $\mathbb{C}P^{N-1}$ if and only if 
	$\langle \xi, \overline{\eta}\rangle=0$.
\end{prop}
\begin{proof}
	Using Lemma 1,  $\psi$ is harmonic from $M$ to $\mathbb{C}P^{N-1}$  if and only if 
	$$\overline{\partial} \xi =-|\xi|^2X+\langle \overline{\partial}X, X\rangle\xi, \quad \partial\eta=-|\eta|^2X+\langle \partial X, X\rangle\eta.$$ Thus we get  $\langle \xi, \overline{\eta}\rangle=0$ from \eqref{eq:3.4}, which completes the proof.
\end{proof}

Write 
  \begin{equation}a:=\frac{|\xi|^2}{\lambda^2}, \quad b:=\frac{|\eta|^2}{\lambda^2}.\label{eq:2.6}\end{equation}
  It is easy to check that $a$ and $b$ are independent of the choice of local lift $X$ and  complex coordinate $z$. Since $0\leq a,b \leq 2$, we can define globally an invariant $\theta:M\rightarrow [0, \pi]$ as follows
  \begin{equation}\theta:=2\arccos(\sqrt{\frac{a}{2}}),\label{eq:2.7}\end{equation}
  which is exactly the \emph{K$\ddot{a}$hler angle} of $\psi$ (cf. \cite{[15]}). Map $\psi$ is called \emph{holomorphic} (resp. \emph{anti-holomorphic, totally real}) if $\theta=0$ (resp. $\pi, \frac{\pi}{2}$).

  Here we introduce two lacal invariants:
  \begin{equation}\Phi:=\frac{\langle \overline{\partial}\xi, \eta\rangle}{\lambda^2} dz,\label{eq:2.8}\end{equation}
  \begin{equation}\Psi:=\langle \partial\xi, \eta\rangle dz^3.\label{eq:2.9}\end{equation}
  It is easy to check that $\Phi$ and $\Psi$ are independent of the choice of local section $X$ and complex coordinate $z$,  i.e., both of them are globally defined on $M$, and we call $\Psi$ the \emph{cubic Hopf differential}.

  Let 
  \begin{equation}S=\left(\begin{array}{cccccccccc} \langle \partial X, X\rangle & 0  \\
  		0 & \langle X, \overline{\partial} X\rangle
  	\end{array}\right). \label{eq:2.10}\end{equation} 
  Then we have
  \begin{equation}\partial(X, \overline{X})=(X, \overline{X}) S+(\xi, \overline{\eta})\label{eq:2.11}\end{equation}
 and 
  \begin{equation}\overline{\partial}(X,
  		\overline{X})=-(X,
  		\overline{X})S^*+(\eta,
  		\overline{\xi}). \label{eq:2.12}\end{equation} 
  From \eqref{eq:2.4}, \eqref{eq:2.11}, \eqref{eq:2.12} and the identity $\partial\overline{\partial}=\overline{\partial}\partial$, we get
  \begin{equation}\overline{\partial} S +\partial S^*-[S, S^*]=\left(\begin{array}{cccccccccc} |\xi|^2-|\eta|^2 & 0  \\
  		0 & |\eta|^2-|\xi|^2
  	\end{array}\right).\label{eq:2.13}\end{equation}

 \section{Totally real flat minimal surfaces in hyperquadric}  
Suppose now that $(M, ds^2)$ is a connected, simply connected domain in complex plane $\mathbb{C}$ with flat metric $ds^2=4dzd\overline{z}$. 
Let $\psi$ be a linearly full totally real
flat minimal immersion from $M$ to $Q_{N-2}$, and $\phi$ is the corresponding map in $G(2,N;\mathbb{R})$. In the following, we say $\phi$ is \emph{totally real} if $\psi$ is totally real.
Using \eqref{eq:2.5} \eqref{eq:2.6} and \eqref{eq:2.7}, condition totally real is equivalent to 
 \begin{equation}|\xi|^2=1=|\eta|^2. \label{eq:3.1}\end{equation} 
From it we obtain following lemma.

 \begin{lem}
  	Suppose $\phi: M\rightarrow G(2,N;\mathbb{R})$ is a totally real immersion. Then locally, there exists some local section $X$ of $\phi$ s.t. $\langle dX, \overline{X}\rangle=0$.
  \end{lem}
  \begin{proof}
  	Let $X$ be any given local section of  $\phi$  satisfies $|X|=1, \ \langle \overline{X}, X\rangle=0$. From the fact that $\phi$ is totally real,  \eqref{eq:2.13}  becomes
  	\begin{equation}\overline{\partial} S +\partial S^*-[S, S^*]=0.\label{eq:3.2}\end{equation}
  	Let $\widetilde{X}$ be  another local section  of  $\phi$ defined by$$(\widetilde{X},
  	\overline{\widetilde{X}})=(X,
  	\overline{X})T$$ where $T\in SO(2)$ is to be determined. To end our proof, we need to prove the existence of some $T\in SO(2)$ satisfies  $\langle \overline{\widetilde{X}}, d\widetilde{X}\rangle=0$.
  	Notice that such $T$ is a solution of the linear PDE
  $$dT+T(Sdz-S^*d\overline{z})=0. $$
  	Its integrable condition is just \eqref{eq:3.2}, so it has a unique solution for any given initial value, which ends our proof.
  \end{proof}

For totally real flat minimal immersion $\phi:M\rightarrow G(2,N;\mathbb{R})$, Proposition 1 together with Lemma 2 implies that, locally, there exists some local section $X$ of $\phi$ s.t. 
  \begin{equation}\overline{\partial} \xi =-X-\langle \xi, \overline{\eta}\rangle\overline{X}=\partial\eta. \label{eq:3.5}\end{equation}
 With it, direct computations give
$$4P=X(\overline{\partial}\eta)^*+\overline{X}(\overline{\partial\xi})^*-\partial\xi X^*-\partial\overline{\eta}(\overline{X})^*, \ 
4P^*=\overline{\partial}\eta X^*+\overline{\partial\xi}(\overline{X})^*- X(\partial\xi)^*-\overline{X}(\partial\overline{\eta})^*.$$
Therefore
$$\|B\|^2=4\tr PP^*=\frac{1}{2}(|\partial\xi|^2+|\overline{\partial}\eta|^2-|\langle \partial\xi, \overline{X}  \rangle|^2-|\langle\overline{\partial}\eta, \overline{X}  \rangle|^2),$$
which can be rewritten as 
\begin{equation}\|B\|^2=\frac{1}{2}(|\pi_{\phi^{\perp}}\partial\xi|^2+|\pi_{\phi^{\perp}}\overline{\partial}\eta|^2)\label{eq:3.32}\end{equation}
by noticing
$$\partial\xi=\pi_{\phi}\partial\xi+\pi_{\phi^{\perp}}\partial\xi=\langle\partial\xi, \overline{X}\rangle\overline{X}+\pi_{\phi^{\perp}}\partial\xi$$ and 
$$\overline{\partial}\eta=\pi_{\phi}\overline{\partial}\eta+\pi_{\phi^{\perp}}\overline{\partial}\eta=\langle\overline{\partial}\eta, \overline{X}\rangle\overline{X}+\pi_{\phi^{\perp}}\overline{\partial}\eta,$$
where $\pi_{\phi}\partial\xi$ and $\pi_{\phi^{\perp}}\partial\xi$ denote the Hermitian orthogonal projections onto $2$-dimensional subbundle $\underline{\phi}$ and $(N-2)$-subbundle $\underline{\phi}^{\perp}$ of the trivial bundel $\underline{\mathbb{C}}^N=M\times \mathbb{C}^N$ respectively.
  Using \eqref{eq:3.32} we obtain

\begin{prop}
	Let $\phi: M\rightarrow G(2,N;\mathbb{R})$ be a linearly full totally real flat minimal immersion. If $\phi$ is totally geodesic, then $N=4$.
\end{prop}
\begin{proof}
From the assumption that $\phi$ is totally geodesic,  
$$\partial \xi=\langle\partial\xi, \overline{X}\rangle\overline{X}=-\langle\xi, \overline{\xi}\rangle\overline{X}, \ 
\overline{\partial}\eta=\langle\overline{\partial}\eta, \overline{X}\rangle\overline{X}=-\langle\eta, \overline{\eta}\rangle\overline{X}$$ holds by using of  \eqref{eq:3.32}. This together with \eqref{eq:3.5} show that $\langle\xi, \overline{\xi}\rangle, \ \langle\eta, \overline{\eta}\rangle, \ \langle\xi, \overline{\eta}\rangle$ are all constants. Then for $\phi_1$ in harmonic sequence 
\begin{equation}\cdots\stackrel{\partial{''}}
	{\longleftarrow} \overline{\underline{\phi}}_{2}\stackrel{\partial{''}} {\longleftarrow}
	\overline{\underline{\phi}}_{1}\stackrel{\partial{''}}
	{\longleftarrow} \underline{\phi} =
	\underline{\phi}_0\stackrel {\partial{'}}
	{\longrightarrow}\underline{\phi}_1  \stackrel{\partial{'}}
	{\longrightarrow}\underline{\phi}_2 \stackrel{\partial{'}}
	{\longrightarrow}\cdots, \label{eq:3.16}\end{equation} 
  it is spanned by $\xi$ and $\overline{\eta}-\langle \overline{\eta}, \xi \rangle\xi$. Then $\underline{\phi}_2={\partial}'\underline{\phi}_1$ is spanned by 
$\partial\xi$ and $\partial(\overline{\eta}-\langle \overline{\eta}, \xi \rangle\xi)$. The later one can be expressed by $$\partial(\overline{\eta}-\langle \overline{\eta}, \xi \rangle\xi)=-\langle \overline{\eta}, \eta \rangle X +\langle \overline{\eta}, \xi \rangle
\langle \xi, \overline{\xi}\rangle \overline{X}.$$

We first claim that $\langle \xi, \overline{\xi}\rangle\neq0$. If not, then $\langle \xi, \overline{\xi}\rangle=0$ holds, which gives
$$\partial \xi=0, \ \langle \eta, \overline{\eta}\rangle\neq0, \ \underline{\phi}_2=\underline{X}.$$
The last relation means that $\underline{X}$ is harmonic, i.e., there exists a harmonic map from $M$ to $\mathbb{C}P^{N+1}$ such that $$\underline{\phi}=\underline{\overline{f}}_{0}\oplus\underline{f}_0, 
\ \underline{\phi}_1=\underline{\overline{f}}_{1}\oplus\underline{f}_{-1}, \ \underline{\phi}_2=\underline{f}_{0}, \ \underline{\phi}_3=\underline{f}_{1}, \underline{\phi}_4=0,$$ where
$\cdots\stackrel{\partial{''}}
{\longleftarrow}\underline{f}_{-2}\stackrel{\partial{''}} {\longleftarrow}
\underline{\eta}=\underline{f}_{-1}\stackrel{\partial{''}}
{\longleftarrow} \underline{X} =
\underline{f}_0\stackrel {\partial{'}}
{\longrightarrow}\underline{\xi}=\underline{f}_1  \stackrel{\partial{'}}
{\longrightarrow}0$ is the harmonic sequence derived by $\underline{f}_0$, which contradicts the fact that  harmonic sequence \eqref{eq:3.16} extends infinitely in both directions.

Using similar discussion gives $\langle \eta, \overline{\eta}\rangle\neq0$. Then we get that $\phi_2=\phi$, which means $N=4$.
\end{proof}

As mentioned in the introduction, this paper is contributed to studying the family of totally real  surfaces that are minimal in both $Q_{N-2}$ and $\mathbb{C}P^{N-1}$. Then we have the following useful characterization.
  
  \begin{theo}
  	Let $\phi: M\rightarrow G(2,N;\mathbb{R})$ be a linearly full totally real isometric harmonic map. If its corresponding immersion  $\psi: M\rightarrow Q_{N-2} \subseteq \mathbb{C}P^{N-1}$ is also harmonic  in  $\mathbb{C}P^{N-1}$, then in  \eqref{eq:3.16}, for $k=0,1,2,...$, \begin{equation}\underline{\phi}_k=\underline{\overline{f}}_{-k}\oplus\underline{f}_k,\label{eq:3.17}\end{equation} where $\underline{f}_k$ are line bundles belonging to the following harmonic sequence in some $\mathbb{C}P^{n}, \ (n \leq N-1)$
\begin{equation}\cdots\stackrel{\partial{''}}
  		{\longleftarrow} \underline{f}_{-2}\stackrel{\partial{''}} {\longleftarrow}
  		\underline{f}_{-1}\stackrel{\partial{''}}
  		{\longleftarrow}  
  		\underline{f}_0\stackrel {\partial{'}}
  		{\longrightarrow}\underline{f}_1  \stackrel{\partial{'}}
  		{\longrightarrow}\underline{f}_2 \stackrel{\partial{'}}
  		{\longrightarrow}\cdots. \label{eq:3.18}\end{equation}
\end{theo}
\begin{proof}
Let $\underline{f}_0=\underline{X}$, 
it follows from the supposition that $\underline{f}_0:M\rightarrow \mathbb{C}P^{n} \ (n \leq N-1)$ is harmonic and generates harmonic sequence  \eqref{eq:3.18} with $\underline{\xi}=\underline{f}_1, \ \underline{\eta}=\underline{f}_{-1}, \ \langle \overline{f}_{-1}, f_{1} \rangle=0$. Then we have
\begin{equation}\underline{\phi}=\underline{\overline{f}}_0\oplus\underline{f}_0, \quad \underline{\phi}_{1}=\underline{\overline{f}}_{-1}\oplus\underline{f}_1.\label{eq:3.19}\end{equation}

In the following we prove \eqref{eq:3.17} by induction on $k$. 
When $k=0, 1$,  the conclusion holds by \eqref{eq:3.19}. Suppose the conclusion is true for $0,1,...,k$. Condider the case of $k+1$. By induction hypotheses we have for $p=0,1,...,k$ $\underline{\phi}_{p}=\underline{\overline{f}}_{-p}\oplus\underline{f}_p$.

Then $\langle \overline{f}_{-k}, f_{k+1} \rangle=0$ can be obtained 
by differentiating with respect to $z$ in $\langle f_{k}, \overline{f}_{-k} \rangle=0$ and using relation $\langle \overline{f}_{k}, f_{-(k-1)} \rangle=0$, which implies \begin{equation}A^{'}_{\phi_k}(f_k):=\pi_{\phi_k^{\bot}}(\partial f_k)=f_{k+1}.\label{eq:3.20}\end{equation}
Similarly 
\begin{equation}A^{'}_{\phi_k}(\overline{f}_{-k}):=\pi_{\phi_k^{\bot}}(\partial \overline{f}_{-k})=\overline{f}_{-(k+1)} \label{eq:3.21}\end{equation} can be verified by formula
$\langle \overline{f}_{-(k+1)}, f_k\rangle=0$, which is given by  
differentiating with respect to $\overline{z}$ in $\langle f_{k}, \overline{f}_{-k} \rangle=0$ and using relation $\langle \overline{f}_{-k}, f_{k-1} \rangle=0$.
$A^{'}_{\phi_k}$ and $A^{''}_{\phi_k}$ shown here  are in fact two vector bundle morphisms  from $\phi_k$ to $\phi_k^{\bot}$, for more details see (cf. \cite{[5]}).

To complish our proof, next we only need to verify that $$\langle f_{k+1}, \overline{f}_{-(k+1)} \rangle=0$$ holds, which can be obtained by differentiating with respect to $\overline{z}$ in $\langle  \overline{f}_{-(k+1)}, f_{k} \rangle=0$. So we finish our proof.
\end{proof}

   \section{Totally real flat minimal surfaces in hyperquadric}  
  In this section,  we regard harmonic maps from $M$ to $G(2,N;\mathbb{R})$ as conformal minimal immersions of $M$ in $G(2,N;\mathbb{R})$.  To analyze them, we need 
  the following result, which appears in \cite{[19]}, see also \cite{[18]} and \cite{[20]}.
   \begin{theo}
  Let $f: M\rightarrow \mathbb{C}P^{n}$ be a linearly full totally real harmonic map with induced metric $f^*ds^2=2dzd\overline{z}$. Then up to a unitary equivalence, $f=[V_{0}^{(n)}]$, where
 \begin{equation}V_{0}^{(n)}(z) = \left(\begin{array}{ccccccc}
e^{z-\overline{z}}\sqrt{r_0} \\
e^{a_1z-\overline{a}_1\overline{z}}\sqrt{r_1} \\
\vdots  \\
e^{a_nz-\overline{a}_n\overline{z}}\sqrt{r_n}
     \end{array}\right),\label{eq:4.1}\end{equation}
   with $r_i > 0$ and $a_i = e^{\sqrt{-1}\theta_i}$ for $i=1,...,n$, satisfying $0<\theta_1<\cdots<\theta_n<2\pi$,  \begin{equation}r_0+r_1+\cdots+r_n=1, \ r_0+\sum_{i=1}^na_ir_i=0, \  r_0+\sum_{i=1}^na_i^2r_i=0.\label{eq:4.300}\end{equation} Here $f$ extends to a totally real harmonic map $f: \mathbb{C}\rightarrow \mathbb{C}P^{n}$.
 \end{theo}
Set $$A=diag(1,a_1,...,a_n),$$ then the harmonic sequence of $f$ is given by 
\begin{equation}\cdots\stackrel{\partial{''}}
{\longleftarrow} \underline{f}_{-2}\stackrel{\partial{''}} {\longleftarrow}
\underline{f}_{-1}\stackrel{\partial{''}}
{\longleftarrow} \underline{f}\stackrel {\partial{'}}
{\longrightarrow}\underline{f}_1  \stackrel{\partial{'}}
{\longrightarrow}\underline{f}_2 \stackrel{\partial{'}}
{\longrightarrow}\cdots, \end{equation} 
where  $f_i=[A^{i}V_{0}^{(n)}]$, for any $i\in \mathbb{Z}$.

  Maps $V_{0}^{(n)}: M\rightarrow S^{2n+1} \subseteq \mathbb{C}^{n+1}$ given by \eqref{eq:4.1} with the first and third equations of \eqref{eq:4.300} satisfied are flat minimal immersions. The second equation of \eqref{eq:4.300} is the condition that the map also be horizontal with respect to the Hopf projection $S^{2n+1} \rightarrow \mathbb{C}P^{n}$. Any totally real flat minimal surface in $\mathbb{C}P^{n}$  is equivalent to a flat horizontal minimal surface in 
  $S^{2n+1}$ (cf. \cite{[90], [19]}).

  Using this result we obtain

  \begin{prop}
  Let $\phi: M\rightarrow G(2,N;\mathbb{R})$ be a linearly full totally real flat minimal immersion. Then $\phi$ is non-degenerate, i.e. rank $\phi_1=$ rank $\phi_{-1}=2$ in \eqref{eq:3.16}.
 \end{prop}
    \begin{proof}
 If $\phi: M\rightarrow G(2,N;\mathbb{R})$ is degenerate, then    \eqref{eq:3.16} becomes
   \begin{equation}\cdots\stackrel{\partial{''}}
{\longleftarrow} \overline{\underline{f}}_{2}\stackrel{\partial{''}} {\longleftarrow}
\overline{\underline{f}}_{1}\stackrel{\partial{''}}
{\longleftarrow} \underline{\phi} =
\underline{\phi}_0\stackrel {\partial{'}}
{\longrightarrow}\underline{f}_1  \stackrel{\partial{'}}
{\longrightarrow}\underline{f}_2 \stackrel{\partial{'}}
{\longrightarrow}\cdots, \label{eq:4.3}\end{equation} 
where $\underline{\overline{f}}_i$ and  $\underline{f}_i, \ i=1,2,...$ are  harmonic line bundles from $M$ to $\mathbb{C}P^{N-1}$, and $\overline{f}_0$ and $f_0$ are two local sections of $ \underline{\phi}$.
From our discussion before, we have $$\underline{\xi}=\underline{f}_{1}, \  \underline{\eta}=\underline{\overline{f}}_{1}.$$
This together with \eqref{eq:3.5} implies that 
$$\underline{\overline{f}}_{0}=\underline{f}_{0},$$ which establishes 
$$\phi=f_{0}\oplus c_{0},$$
where $\underline{c}_{0}$ is a constant in $\mathbb{C}^{N}$. 

Since $\phi=f_{0}\oplus c_{0}:M\rightarrow G(2,N;\mathbb{R})$ is a totally real  flat harmonic map, then $\underline{f}_{0}:M\rightarrow \mathbb{C}P^{N-2}$ is also a  totally real  flat harmomic map. By using of Theorem 2, there exists some unitary matrix $U \in U(N-1)$ such that 
$$f_{0}=UV_{0}^{(N-2)},$$ which shows
\begin{equation}\left(\begin{array}{ccccccc}
e^{\overline{z}-z}\sqrt{r_0} \\
e^{\overline{a}_1\overline{z}-a_1z}\sqrt{r_1} \\
\vdots  \\
e^{\overline{a}_{n}\overline{z}-a_{n}z}\sqrt{r_{n}}
     \end{array}\right)=U^TU\left(\begin{array}{ccccccc}
e^{z-\overline{z}}\sqrt{r_0} \\
e^{a_1z-\overline{a}_1\overline{z}}\sqrt{r_1} \\
\vdots  \\
e^{a_{n}z-\overline{a}_{n}\overline{z}}\sqrt{r_{n}}
     \end{array}\right),\label{eq:4.4}\end{equation}
      where $n=N-2$.
     Then set \begin{equation}W=U^TU=(w_{ij}),\label{eq:4.5}\end{equation}
    we have
\begin{equation}W^T=W, \quad  W^*W=I\label{eq:4.6}\end{equation} holds.  
From   \eqref{eq:4.4}      we obtain
     $$\sqrt{r_0}=w_{00}\sqrt{r_0}e^{2z-2\overline{z}}+w_{01}\sqrt{r_1}e^{(a_{1}+1)z-(\overline{a}_{1}+1)\overline{z}}+\cdots+w_{0n}\sqrt{r_n}e^{(a_{n}+1)z-(\overline{a}_{n}+1)\overline{z}},$$ which gives
  $$\left(\begin{array}{ccccccc}
2 & a_{1}+1 & \cdots & a_{n}+1 \\
2^{2} & (a_{1}+1)^{2} & \cdots & (a_{n}+1)^{2} \\
\vdots & \vdots &  &\vdots  \\
2^{n} & (a_{1}+1)^{n} & \cdots & (a_{n}+1)^{n}
     \end{array}\right)
     \left(\begin{array}{ccccccc}
w_{00}\sqrt{r_0} \\
w_{01}\sqrt{r_1} \\
\vdots  \\
w_{0n}\sqrt{r_n}
     \end{array}\right)
    =\left(\begin{array}{ccccccc}
0 \\
0 \\
\vdots  \\
0
     \end{array}\right).$$
     Then $w_{00}=w_{01}=\cdots=w_{0n}=0$ can be obtained since the determinant of 
     coefficient matrix of the above equation is
     $$2[\prod_{i=1}^{n}(a^2_i-1)][\prod_{1\leq j<i \leq n}(a_i-a_j)]\neq 0,$$
     which contradicts the fact $U\in U(N-1)$. This finishes the proof.
  \end{proof}

  	From now on we consider $\phi=\overline{X}\oplus X: \mathbb{C} \rightarrow G(2,N;\mathbb{R})$ as a linearly full totally real flat minimal immersion with induced metric $\phi^*ds^2=4dzd\overline{z}$, and $\underline{X}: \mathbb{C} \rightarrow \mathbb{C}P^{N-1}$ is also minimal. It follows from Theorem 1 that, locally, there exists a locally totally real minimal immersion ${\underline f}_0:\mathbb{C} \rightarrow \mathbb{C}P^{N-1}$ with induced metric  $f_0^*ds^2=2dzd\overline{z}$ such that
  	$$\phi=\overline{f}_0\oplus f_0.$$
  	Then by Theorem 2, there exists an unitary matrix $U \in U(N)$ such that $$f_0=UV_0^{(n)}$$ for some $1 \leq n \leq N-1$ (adding zeros to the end of $V^{(n)}_0$ such that it belongs to $\mathbb{C}^N$, in the absence of confusion, we also denote it by $V^{(n)}_0$). 
Then relation   $\langle f_0,\overline{f}_0\rangle=0$ becomes $\langle UV^{(n)}_0,\overline{U}\overline{V}^{(n)}_0\rangle=0$, which is equivalent to
  	\begin{equation}\tr WV^{(n)}_0V^{(n)T}_0=0.\label{eq:4.7}\end{equation} Through the expression of $V^{(n)}_0$ given in \eqref{eq:4.1}, set $a_0=1$, \eqref{eq:4.7} becomes 
  		\begin{equation}\sum_{i,j=0}^n(w_{ij}\sqrt{r_ir_j}e^{(a_i+a_j)z-(\overline{a}_i+\overline{a}_j)\overline{z}})=0.\label{eq:4.8}\end{equation}
Here $W$ satisfies \eqref{eq:4.5} and \eqref{eq:4.6}.  Define a set 
 $$G_W\triangleq\{U\in U(N)\rvert U^TU=W\}.$$
 For a given $W$, the following can be checked easily,
 \\(i) $\forall$ $O\in SO(N), U \in G_W$, we have that $OU\in G_W$;
 \\(ii) $\forall U, V\in G_W, \exists O\in SO(N), s.t. U=OV$. 
  
  Set $a_0=1$, in the following we analyse the conditions of $f_0$ in the linearly full case in more detail by 
  discussing $W$ in three cases respectively.
  
  {\bfseries Case I: $\forall i,j=0,1,2,...,n, a_i+a_j\neq 0$.} 
  
  Here we first claim that $\forall i,j,k,l=0,1,2,...,n$, if $\{i,j\}\neq\{k,l\}$, following important formula
  \begin{equation}a_i+a_j\neq a_k+a_l\label{eq:4.9}\end{equation} holds. Otherwise if $a_i+a_j=a_k+a_l$, it can be transformed as
  $$\cos \theta_i+\cos \theta_j=\cos \theta_k+\cos \theta_l, \quad \sin \theta_i+\sin \theta_j=\sin \theta_k+\sin \theta_l,$$
  which gives
  \begin{equation}\cos(\theta_i- \theta_j)=\cos(\theta_k- \theta_l).\label{eq:4.10}\end{equation}
  Without loss of generality, set $\theta_j \leq \theta_i,\ \theta_l \leq \theta_k$, then \eqref{eq:4.10} establishes 
 either $\frac{a_i}{a_j}=\frac{a_k}{a_l}$ or $\frac{a_i}{a_j}=\frac{a_l}{a_k}$, which is impossible. So \eqref{eq:4.9} holds.

Observing  \eqref{eq:4.8} we find  $\forall p \in \{0,1,2,...\}$,
  	  $$\sum_{i,j=0}^nw_{ij}\sqrt{r_ir_j}(a_i+a_j)^p=0.$$ 
This together with \eqref{eq:4.9} 	gives  
 $w_{ij}=0, \forall i,j=0,1,2,...,n$,
  	which implies that matrix $W$ is of the following type
  	\begin{equation}
  	W=\left(\begin{array}{ccccccc}
  	&  &  & w_{0,n+1} & \cdots & w_{0,N-1} \\
  	& \text{{\huge{0}}} &  & \vdots & & \vdots \\
  	&  &  & w_{n,n+1} & \cdots & w_{n,N-1} \\
  	w_{0,n+1} & \cdots & w_{n,n+1} & w_{n+1,n+1} & \cdots & w_{n+1,N-1}
  	\\ \vdots & & \vdots & \vdots & & \vdots
  	\\w_{0,N-1} & \cdots & w_{n,N-1}  & w_{n+1,N-1} & \cdots & w_{N-1, N-1}
  	\end{array}\right).\label{eq:4.11}\end{equation}
Therefore for all $k \in \mathbb{Z}$,  \eqref{eq:4.11} gives
\begin{equation}\langle f_0,\overline{f}_k\rangle=0.\label{eq:4.12}\end{equation} Then we have $$N=2n+2$$ since  $\dim($span$\{f_k\}_{k\in \mathbb{Z}}\oplus\{\overline{f}_k\}_{k\in \mathbb{Z}})=2n+2$ by \eqref{eq:4.12}.	
Choose \begin{equation} U_1=\left(\begin{array}{ccccccc}
  	\frac{1}{\sqrt2} & 0 & 0 & \cdots & 0 & \\
  	\frac{\sqrt{-1}}{\sqrt2} & 0 & 0 & \cdots & 0 & \\
  	0 & \frac{1}{\sqrt2} & 0 & \cdots & 0 &  \\
  	0 & \frac{\sqrt{-1}}{\sqrt2} & 0 & \cdots & 0 & \\
  	0 & 0 & \frac{1}{\sqrt2} & \cdots & 0 & \text{{\huge{*}}} \\
  	0 & 0 & \frac{\sqrt{-1}}{\sqrt2} & \cdots & 0 & \\
  	\vdots & \vdots & \vdots & \ddots & \vdots & \\
  	0 & 0 & 0 & \cdots & \frac{1}{\sqrt2} & \\
  	0 & 0 & 0 & \cdots & \frac{\sqrt{-1}}{\sqrt2} &
  	\end{array}\right).\label{eq:4.13}\end{equation}
  	Let $C_i$ denote the $i$-th column of $U_1$.  For $1\leq i\leq n+1$,
  	all the elements of $C_i$ are zero expect the $(2i-1)$-th and
  	$(2i)$-th elements, which are $\frac{1}{\sqrt2}$ and
  	$\frac{\sqrt{-1}}{\sqrt 2}$ respectively.
  	From the above discussion, we have,  up to an isometry
  	of $G(2,N;\mathbb{R})$,
  	$\underline{\phi}=\overline{\underline{f}}_0\oplus
  	\underline{f}_0$ with
  	\begin{equation}f_0=U_1V^{(n)}_0=\frac{1}{\sqrt{2}}\left(\begin{array}{ccccccc}
  	e^{z-\overline{z}}\sqrt{r_0} \\e^{z-\overline{z}}\sqrt{-r_0} \\
  	e^{a_1z-\overline{a}_1\overline{z}}\sqrt{r_1} \\e^{a_1z-\overline{a}_1\overline{z}}\sqrt{-r_1} \\
  	\vdots  \\
  	e^{a_nz-\overline{a}_n\overline{z}}\sqrt{r_n}\\
  	e^{a_nz-\overline{a}_n\overline{z}}\sqrt{-r_n}
  	\end{array}\right).\label{eq:4.23}\end{equation}

  	{\bfseries Case II: $\forall i, \exists j  \ s.t. a_i+a_j = 0$.} 
	
	Firstly we notice that in this case for any $i$, there exists and only exists one $j$ such that $a_i+a_j = 0$. Hence \emph{$n$ must be odd}. Set $n=2m+1$. Then for any $i, j, k, l=0,1,...,m$, if $\{i,j\}\neq\{k,l\}$, 
  	\begin{equation}a_i+a_j\neq \pm(a_k+a_l), \ a_i+a_{j+m+1}\neq a_k+a_{l+m+1}, \ a_i+a_j\neq \pm(a_k+a_{l+m+1}). \label{eq:4.14}\end{equation} 
Observing  \eqref{eq:4.8} and using relation $a_i+a_{i+m+1}=0, \forall i=0,1,2,...,m$, we find  $\forall p \in \{0,1,2,...\}$,
$$ \begin{array}{lll}
0 & = & \sum_{i,j=0}^m[w_{ij}\sqrt{r_ir_j}(a_i+a_j)^p+w_{i+m+1,j+m+1}\sqrt{r_{i+m+1}r_{j+m+1}}(a_{i+m+1}+a_{j+m+1})^p]+ \\
& &\sum_{i,j=0, i\neq j}^mw_{i,j+m+1}\sqrt{r_ir_{j+m+1}}(a_i+a_{j+m+1})^p.\end{array}
$$
This together with \eqref{eq:4.14} 	gives  
$$w_{ij}=0, \  w_{i+m+1,j+m+1}=0, \forall i,j=0,1,2,...,m$$ and 
$$w_{i,j+m+1}=0, \forall i,j=0,1,2,...,m, i\neq j.$$
Then matrix $W$ is of the following type  	
\begin{equation} W=\left(\begin{array}{cccccccccc}
0 & \cdots & 0 & w_{0,m+1} & \cdots & 0  & w_{0,n+1} & \cdots & w_{0,N-1}\\
\cdots & \ddots & \cdots & \cdots & \ddots & \cdots  &\cdots & \ddots & \cdots\\
0 & \cdots & 0 & 0 & \cdots & w_{m,n}  & w_{m,n+1} & \cdots & w_{m,N-1}\\
w_{0,m+1} & \cdots & 0  & 0 & \cdots & 0  & w_{m+1,n+1} & \cdots & w_{m+1,N-1}\\
\cdots & \ddots & \cdots & \cdots & \ddots & \cdots  &\cdots & \ddots & \cdots\\
0 & \cdots & w_{m,n}  & 0 & \cdots & 0  & w_{n,n+1} & \cdots & w_{n,N-1}\\
w_{0,n+1} & \cdots & w_{m,n+1}  & w_{m+1,n+1} & \cdots & w_{n,n+1}  & w_{n+1,n+1} & \cdots & w_{n+1,N-1}\\
\cdots & \ddots & \cdots & \cdots & \ddots & \cdots  &\cdots & \ddots & \cdots\\
w_{0,N-1} & \cdots & w_{m,N-1}  & w_{m+1,N-1} & \cdots & w_{n,N-1}  & w_{n+1,N-1} & \cdots & w_{N-1,N-1}
\end{array}\right).\label{eq:4.15}\end{equation}
With it, for any nonnegative integers $i,j$,  direct computations give
\begin{equation} 
	\langle f_i, \overline{f}_j \rangle=[(-1)^i+(-1)^j](\sum_{k=0}^{m}a_k^{i+j}\sqrt{r_kr_{m+1+k}}w_{k,m+1+k}).
	\label{eq:4.20}\end{equation} 
	From it we have
	\begin{equation} 
	\langle f_0, \overline{f}_j \rangle=[1+(-1)^j](\sum_{k=0}^{m}a_k^{j}\sqrt{r_kr_{m+1+k}}w_{k,m+1+k})
	\label{eq:4.21}\end{equation} and
\begin{equation} 
\sum_{k=0}^{m}\sqrt{r_kr_{m+1+k}}w_{k,m+1+k}=0.
\label{eq:4.22}\end{equation}

Suppose $n=N-1$. Then $|w_{i,i+m+1}|=1, i=0,1,...,m$. In this case, choose $\sqrt{2}U$ of following type
\begin{equation}\left(\begin{array}{cccccccccc}
		\sqrt{w_{0,m+1}} & 0 & \cdots & 0 & \sqrt{w_{0,m+1}} & 0  &  \cdots & 0\\
		0 & \sqrt{w_{1,m+2}} & \cdots & 0 & 0 & \sqrt{w_{1,m+2}}  &\cdots & 0 \\
		\vdots & \vdots & \vdots & \vdots & \vdots & \vdots  &\vdots & \vdots \\
		0 & 0 & \cdots & \sqrt{w_{mn}} & 0  & 0 & \cdots & \sqrt{w_{mn}}\\
		0 & 0 & \cdots & \sqrt{-w_{mn}} & 0  & 0 & \cdots & -\sqrt{-w_{mn}}\\
		\vdots & \vdots & \vdots & \vdots & \vdots & \vdots  &\vdots & \vdots \\
		0 & \sqrt{-w_{1,m+2}} & \cdots & 0 & 0 & -\sqrt{-w_{1,m+2}}  &\cdots & 0 \\
		\sqrt{-w_{0,m+1}} & 0 & \cdots & 0 & -\sqrt{-w_{0,m+1}} & 0  &  \cdots & 0
	\end{array}\right)\label{eq:4.16}\end{equation}
   	and therefore
\begin{equation}f_0=UV^{(n)}_0=\frac{1}{\sqrt{2}}\left(\begin{array}{ccccccc}
\sqrt{w_{0,m+1}}(e^{z-\overline{z}}\sqrt{r_0}+e^{\overline{z}-z}\sqrt{r_{m+1}})
 \\
 \sqrt{w_{1,m+2}}(e^{a_1z-\overline{a}_1\overline{z}}\sqrt{r_1}+e^{\overline{a}_1\overline{z}-a_1z}\sqrt{r_{m+2}})
\\
\vdots  
\\
\sqrt{w_{mn}}(e^{a_mz-\overline{a}_m\overline{z}}\sqrt{r_m}+e^{\overline{a}_m\overline{z}-a_mz}\sqrt{r_{n}})
\\
\sqrt{-w_{mn}}(e^{a_mz-\overline{a}_m\overline{z}}\sqrt{r_m}-e^{\overline{a}_m\overline{z}-a_mz}\sqrt{r_{n}})
\\
\vdots
\\
\sqrt{-w_{1,m+2}}(e^{a_1z-\overline{a}_1\overline{z}}\sqrt{r_1}-e^{\overline{a}_1\overline{z}-a_1z}\sqrt{r_{m+2}})
\\
\sqrt{-w_{0,m+1}}(e^{z-\overline{z}}\sqrt{r_0}-e^{\overline{z}-z}\sqrt{r_{m+1}})
\end{array}\right).\label{eq:4.50}\end{equation}

{\bfseries Case III: $\exists i, \forall j  \ s.t. a_i+a_j \neq 0$ and $\exists k, l \ s.t. a_k+a_l=0$.} 

 In this case, suppose $a_{i_\alpha}+a_{j_\alpha}=0, \ 0\leq i_\alpha<j_\alpha\leq n$ for $\alpha=1,...,s$ and for $\beta=2s+1,...,n+1, \forall j, a_{i_\beta}+a_{j}\neq 0, 0\leq i_\beta \leq n$. Then except for $w_{i_1j_1}, w_{i_2j_2},..., w_{i_sj_s}$ we have 
$w_{ij}=0, i, j=0,1,2,...,n$, and \eqref{eq:4.8} becomes
\begin{equation}
 w_{i_1j_1}\sqrt{r_{i_1}r_{j_1}}+w_{i_2j_2}\sqrt{r_{i_2}r_{j_2}}+\cdots+w_{i_sj_s}\sqrt{r_{i_s}r_{j_s}}=0.\label{eq:4.51}\end{equation}

 Next we shall illustrate the discussion of cases (I), (II) and (III) for the type of matrix $W$ in order to give classifications of linear full totally real flat minimal immersions of $\mathbb{C}$ in $Q_2, Q_3$ and $Q_4$.

  \begin{prop}
  	For any linearly full totally real flat immersion  $f: \mathbb{C} \rightarrow Q_2\subseteq \mathbb{C}P^3$, if it is minimal in both $Q_2$ and $\mathbb{C}P^3$, then up to $SO(4)$ equivalence, \begin{equation}f(z)=\left[\left(\begin{array}{ccccccc}
  	e^{z-\overline{z}}+e^{\overline{z}-z}
  	\\
  	\sqrt{-1}(e^{\sqrt{-1}(z+\overline{z})}+e^{-\sqrt{-1}(z+\overline{z})})
  	\\
  	e^{\sqrt{-1}(z+\overline{z})}-e^{-\sqrt{-1}(z+\overline{z})}
  	\\
  	\sqrt{-1}(e^{z-\overline{z}}-e^{\overline{z}-z})
  	\end{array}\right)\right].\label{eq:4.17}\end{equation}
  \end{prop}
  \begin{proof}
  	Set $f=UV^{(n)}_0$ for some $U \in U(4)$, then from Theorem 2 we notice that $2\leq n\leq 3$, which shows  only cases II and III may happen here.
  	
  	If $f$ belongs to case II, then $m=1, \ n=3$ holds. In this case, from our analysis above, $$a_0=1, \ a_2=-1, \ a_1+a_3=0, \ W=\left(\begin{array}{cccccccccc}
  	0 & 0 &   w_{02} & 0\\
  	0 & 0 &   0 & w_{13}\\
  	w_{02} &  0 & 0 & 0\\
  	0 & w_{13}  & 0 & 0
  	\end{array}\right)$$ 
  	with $|w_{02}|=1, \ |w_{13}|=1$. This together with \eqref{eq:4.300} and \eqref{eq:4.22}
  implies $$r_0=r_1=r_2=r_3=\frac{1}{4}, \ a_1=\sqrt{-1}, \ a_3=-\sqrt{-1}, \ w_{02}+w_{13}=0.$$
  	Then we set $ w_{02}=1, \ w_{13}=-1$, and choose $U=\frac{1}{\sqrt{2}}\left(\begin{array}{cccccccccc}
  1 & 0 & 1 & 0 \\
  0 & \sqrt{-1} & 0 & \sqrt{-1} \\
 0&1 & 0 & -1 \\ \sqrt{-1} & 0 & -\sqrt{-1} & 0
  \end{array}\right)$, immersion $f$ is the one given in \eqref{eq:4.17}, which is in fact the Clifford solution (see in section 5).
  
  If $f$ belongs to case III. Since there exists some $i_{\beta}$ s.t. $w_{i_{\beta}k}=0$ for $k=0,1,...,n$. Then $n = 2$, $s=1$ and 
  for $k=0,1,2,3$, 
  $$w_{i_1j_1}=0, \quad
  |w_{3k}|=\left\{\begin{array}{rcl}1, & & {k=i_{\beta}}\\0, & & {k\neq i_{\beta}}.\end{array}\right.$$ 
  The first equation is an immediately result of \eqref{eq:4.51}. Therefore elements of  both the ($i_1+1$)-th and ($j_1+1$)-th columns of $W$ are zeros, which contradicts $W \in U(4)$.  In summary, we finish our proof.
  	\end{proof}

  \begin{prop}
  	Let $f: \mathbb{C} \rightarrow Q_3$ be a linearly full totally real flat minimal immersion. Then $f$ is not minimal in $\mathbb{C}P^4$.
  \end{prop}
  \begin{proof}
  	Suppose there exists some linearly full totally real flat immersion $f: \mathbb{C} \rightarrow Q_3 \subseteq \mathbb{C}P^4$, which is   minimal in both $Q_3$ and $\mathbb{C}P^4$. Then from Theorem 2, 
  $f=UV^{(n)}_0$ for some $U \in U(5), \ 2\leq n\leq 4$, and only cases II and III may happen here.
  	
  	If $f$ belongs to case II, then $m=1, \ n=3$ holds. In this case, $W$ can be rewritten as $$W=\left(\begin{array}{cccccccccc}
  	0 & 0 &   w_{02} & 0 & 0\\
  	0 & 0 &  0 &  w_{13} & 0\\
  	w_{02} & 0 & 0 & 0 & 0\\
  	0 & w_{13} &   0 & 0 & 0 \\
  	0 & 0 &  0 & 0 & w_{44} 
  	\end{array}\right)$$  by \eqref{eq:4.15} and
  	using the property of unitary matirx, which implies that immersion $f$ is not linearly full.
  	
  	If $f$ belongs to case III. Since there exists some $i_{\beta}$ s.t. $w_{i_{\beta}k}=0$ for $k=0,1,...,n$. Then $n \leq 3$ and therefore $s=1$. So from  \eqref{eq:4.51} we get $w_{ij}=0$ for any $i,j=0,1,...,n$, which is impossible.
This completes the proof.
  \end{proof}

\begin{prop}
	There exist continuous families of noncongruent linearly full totally real flat minimal immersions from $\mathbb{C}$ to $Q_4$, which are also minimal in $\mathbb{C}P^5$.
\end{prop}
\begin{proof}
	Let $f: \mathbb{C} \rightarrow Q_4 \subseteq \mathbb{C}P^5$ be any linearly full totally real flat  immersion, which is minimal in both $Q_4$ and $\mathbb{C}P^5$. Here we need to analyze  $f$ by cases I, II and III respectively.

	If $f$ belongs to case I. Then $n=2$ holds. In this case,  by \eqref{eq:4.300}, the Clifford solution (see in section 5) is, up to congruence, the only totally real flat minimal immersion when $n=2$. That is $$r_0=r_1=r_2=\frac{1}{3}, \ a_0=1, \ a_1=e^{\frac{2\pi\sqrt{-1}}{3}}, \ a_2=e^{\frac{4\pi\sqrt{-1}}{3}}$$
	and from \eqref{eq:4.23}
		\begin{equation}f=\frac{1}{\sqrt{6}}\left(\begin{array}{ccccccc}
	e^{z-\overline{z}} \\\sqrt{-1}e^{z-\overline{z}} \\
	e^{a_1z-\overline{a}_1\overline{z}} \\\sqrt{-1}e^{a_1z-\overline{a}_1\overline{z}} \\
	e^{a_2z-\overline{a}_2\overline{z}}\\
	\sqrt{-1}e^{a_2z-\overline{a}_2\overline{z}}
	\end{array}\right)\label{eq:4.53}\end{equation}
with $a_1=e^{\frac{2\pi\sqrt{-1}}{3}}, \ a_2=e^{\frac{4\pi\sqrt{-1}}{3}}$.	
	
If $f$ belongs to case II. Then we have $m=1, n=3$ or $m=2, n=5$. When  $m=2, n=5$, by \eqref{eq:4.15} and \eqref{eq:4.22} we have $\sqrt{r_0r_3}w_{03}+ \sqrt{r_1r_4}w_{14}+\sqrt{r_2r_5}w_{25}=0$, 
$$W=\left(\begin{array}{cccccccccc}
0 & 0 & 0 &  w_{03} & 0 & 0\\
0 & 0 &  0 & 0 &  w_{14} & 0\\
0 & 0 &  0 & 0 &  0 & w_{25}\\
w_{03} & 0 & 0 & 0 & 0 & 0\\
0 & w_{14} &   0 & 0 & 0 & 0\\
0 & 0 &  w_{25} & 0 & 0 & 0
\end{array}\right)$$
and 
\begin{equation}f=\frac{1}{\sqrt{2}}\left(\begin{array}{ccccccc}
\sqrt{w_{03}}(e^{z-\overline{z}}\sqrt{r_0}+e^{\overline{z}-z}\sqrt{r_{3}})
\\
\sqrt{w_{14}}(e^{a_1z-\overline{a}_1\overline{z}}\sqrt{r_1}+e^{\overline{a}_1\overline{z}-a_1z}\sqrt{r_{4}})  
\\
\sqrt{w_{25}}(e^{a_2z-\overline{a}_2\overline{z}}\sqrt{r_2}+e^{\overline{a}_2\overline{z}-a_2z}\sqrt{r_{5}})
\\
\sqrt{-w_{25}}(e^{a_2z-\overline{a}_2\overline{z}}\sqrt{r_2}-e^{\overline{a}_2\overline{z}-a_2z}\sqrt{r_{5}})
\\
\sqrt{-w_{14}}(e^{a_1z-\overline{a}_1\overline{z}}\sqrt{r_1}-e^{\overline{a}_1\overline{z}-a_1z}\sqrt{r_{4}})
\\
\sqrt{-w_{03}}(e^{z-\overline{z}}\sqrt{r_0}-e^{\overline{z}-z}\sqrt{r_{3}})
\end{array}\right).\label{eq:4.70}\end{equation}
Otherwise when $m=1, n=3$, using the property of unitary matrix, $W$ can be expressed by 
$$W=\left(\begin{array}{cccccccccc}
0 & 0 & w_{02} &  0 & w_{04} & w_{05}\\
0 & 0 &  0 & w_{13} &  w_{14} & w_{15}\\
w_{02} & 0 &  0 & 0 &  w_{24} & w_{25}\\
0 & w_{13} & 0 & 0 & w_{34} & w_{35}\\
w_{04} & w_{14} &   w_{24} & w_{34} & w_{44} & w_{45}\\
w_{05} & w_{15} &   w_{25} & w_{35} & w_{45} & w_{55}
\end{array}\right)
=\left(\begin{array}{cccccccccc}
0 & 0 & w_{02} &  0 & 0 & 0\\
0 & 0 &  0 & w_{13} &  0 & 0\\
w_{02} & 0 &  0 & 0 &  0 & 0\\
0 & w_{13} & 0 & 0 & 0 & 0\\
0 & 0 &   0 & 0 & w_{44} & w_{45}\\
0 & 0 &   0 & 0 & w_{45} & w_{55}
\end{array}\right),$$ which means that $f$ is not linearly full in this condition.

If $f$ belongs to case III. Then we have $2s+1\leq n+1 < N=6$, which implies
$s=1, n=2$ or $s=1, n=3$ or $s=1, n=4$ or $s=2, n=4$. When $s=1$, \eqref{eq:4.51} shows $w_{i_1j_1}=0$, then only $s=1, n=2$ holds and $f$ is the one in \eqref{eq:4.53}, which is impossible since there does not exist $i_1, j_1$ such that $a_{i_1}+a_{j_1}=0$. When $s=2, n=4$, there are $15$ conbinations for $i_1, j_1, i_2, j_2 \in {0,1,2,3,4}$. In fact, for each conbination, there exists continuous families of noncongruent $f$. Take $\{i_1, j_1, i_2, j_2\}=\{0, 2, 1, 3\}$ for example, which gives
$$W=\left(\begin{array}{cccccccccc}
0 & 0 & w_{02} &  0 & 0 & 0\\
0 & 0 &  0 & w_{13} &  0 & 0\\
w_{02} & 0 &  0 & 0 &  0 & 0\\
0 & w_{13} & 0 & 0 & 0 & 0\\
0 & 0 &   0 & 0 & 0 & w_{45}\\
0 & 0 &   0 & 0 & w_{45} & 0
\end{array}\right).$$
From \eqref{eq:4.51}, $w_{02}+w_{13}=0$ holds. Without loss of generality, it can be given by
$$W=\left(\begin{array}{cccccccccc}
0 & 0 & 1 &  0 & 0 & 0\\
0 & 0 &  0 & -1 &  0 & 0\\
1 & 0 &  0 & 0 &  0 & 0\\
0 & -1 & 0 & 0 & 0 & 0\\
0 & 0 &   0 & 0 & 0 & t\\
0 & 0 &   0 & 0 & t & 0
\end{array}\right),$$ and 
\begin{equation}f=\frac{1}{\sqrt{2}}\left(\begin{array}{ccccccc}
e^{z-\overline{z}}\sqrt{r_0}+e^{\overline{z}-z}\sqrt{r_{2}}
\\
e^{z-\overline{z}}\sqrt{-r_0}-e^{\overline{z}-z}\sqrt{-r_{2}}
\\
e^{a_1z-\overline{a}_1\overline{z}}\sqrt{r_1}-e^{\overline{a}_1\overline{z}-a_1z}\sqrt{r_{3}}
\\
e^{a_1z-\overline{a}_1\overline{z}}\sqrt{-r_1}+e^{\overline{a}_1\overline{z}-a_1z}\sqrt{-r_{3}}
\\
\sqrt{t}e^{a_4z-\overline{a}_4\overline{z}}\sqrt{r_4}
\\
\sqrt{-t}e^{a_4z-\overline{a}_4\overline{z}}\sqrt{r_4}
\end{array}\right),\label{eq:4.60}\end{equation}
where $t$ is parameters. 	Summing up, $f$ is \eqref{eq:4.53},  \eqref{eq:4.70} or some types of \eqref{eq:4.60} by changing combinations.
\end{proof}

 \section{Clifford Solutions}   
 \begin{definition}
 	\emph{For any $n \geq 2$, a solution to \eqref{eq:4.1} is given by $r_i=\frac{1}{n+1}$ and $a_i=e^{\frac{2\pi \sqrt{-1}i}{n+1}}$ for $i=0,1,...,n$. We will call this the} Clifford solution \emph{to \eqref{eq:4.1}. For such Clifford solution, if there exists an unitary matrix $U\in U(N)$ s.t. $\langle UV^{(n)}_0, \overline{UV}^{(n)}_0\rangle=0$, then we will also call $\underline{UV}^{(n)}_0:\mathbb{C} \rightarrow Q_{N-2}$ a} Clifford solution\emph{.}
  \end{definition}   
 It follows from Proposition 4.1 in \cite{[19]} that,  if $\underline{V}^{(n)}_0$  is a Clifford solution, then, for any $z\in \mathbb{C}$, the vectors $V^{(n)}_0(z), AV^{(n)}_0(z), ..., A^nV^{(n)}_0(z)$ form a unitary set in $\mathbb{C}^{n+1}$, and the harmonic sequence of $\underline{V}^{(n)}_0$ is given by $\underline{V}^{(n)}_i$, where $V^{(n)}_i=A^iV^{(n)}_0$ for any $i\in \mathbb{Z}$. By using $A^{n+1}=I$, we see that the sequence is cyclic (cf. \cite{[21]} for details and additional references). Moreover, up to congruence, it defines the unique totally real flat harmonic map  $\underline{V}^{(n)}_0:\mathbb{C} \rightarrow \mathbb{C}P^n$ that is cyclic.
    
Then by Definition 1, \eqref{eq:4.17} and \eqref{eq:4.53}  are two  Clifford solutions. Furthermore Proposition 5 tells us that, the Clifford solution  \eqref{eq:4.17} is, up to congruence, the only linearly full totally real flat minimal immersion both in $Q_2$ and   $\mathbb{C}P^3$.
    
 Let $\phi:\mathbb{C} \rightarrow G(2,N;\mathbb{R})$ be a linearly full totally real flat minimal surface with induced metric $\phi^*ds^2=4dzd\overline{z}$. If its corresponding immersion $\psi$ in $\mathbb{C}P^{N-1}$ is also minimal, then we can find some $U\in U(N)$ such that $f_0=UV^{(n)}_0$ and $\phi=\overline{f}_0 \oplus f_0$.  From now on, we shall consider when the solutions are acturally Clifford solutions, that is, 
 $f_0=UV^{(n)}_0$ is a Clifford solution.  
 
 To give an explicit characterization of these Clifford solutions, we need to analyze harmonic maps $\underline{f}_0:\mathbb{C} \rightarrow Q_{N-2}$ by cases I, II and III given in section 4 respectively. To do this, the most important step is to discuss matrix $W$ by using the fact that the harmonic sequence $\underline{f}_i=\underline{UV}^{(n)}_i$ derived by $\underline{f}_0$  is cyclic. That is for any $i,j=0,1,...,n$, 
 \begin{equation}\langle f_i, f_j \rangle=\delta_{ij}, \quad f_{n+1+i}=f_i.  \label{eq:5.1}\end{equation}

{\bfseries Case I: $\forall i,j=0,1,2,...,n, a_i+a_j\neq 0$.} 

Since $\underline{f}_0$ is a Clifford solution, then $r_0=r_1=\cdots=r_n=\frac{1}{n+1}$. According to the discussion in section 4 and \eqref{eq:4.23}, $2n+2=N$.
Up to $SO(N)$ equivalence, $f_0$ can only be expressed by  
\begin{equation}f_0=U_1V^{(n)}_0=\frac{1}{\sqrt{N}}\left(\begin{array}{ccccccc}
e^{z-\overline{z}} \\e^{z-\overline{z}}\sqrt{-1} \\
e^{a_1z-\overline{a}_1\overline{z}} \\e^{a_1z-\overline{a}_1\overline{z}}\sqrt{-1} \\
\vdots  \\
e^{a_nz-\overline{a}_n\overline{z}}\\
e^{a_nz-\overline{a}_n\overline{z}}\sqrt{-1}
\end{array}\right),\label{eq:5.2}\end{equation}
where $a_i=e^{\frac{2\pi \sqrt{-1}i}{n+1}}, i=0,1,2,...,n$ are the $n+1$-roots of unity.

 {\bfseries Case II: $\forall i, \exists j  \ s.t. a_i+a_j = 0$.} 
 
 In this case we have already known that $n$ is odd. Set $n=2m+1$.
Since $\underline{f}_0$ is a Clifford solution, then, for any $z\in \mathbb{C}$, vectors $\overline{f}_0(z), \overline{f}_1(z),..., \overline{f}_n(z)$ and $f_0(z), f_1(z),..., f_n(z)$ form two unitary sets in 
$\mathbb{C}^{n+1}$ separately. Notice that $\mathbb{C}^{N}$ is spanned by
$\overline{f}_0(z), \overline{f}_1(z),..., \overline{f}_n(z), f_0(z), f_1(z),..., f_n(z)$. Set  $(2n+2)\times (2n+2)$-matrix 
$$R=\left(\begin{array}{cccccccccc}
E & B \\
C & D \end{array}\right),$$  where  $E=(e_{ij}),  B=(b_{ij}), C=(c_{ij}), D=(d_{ij})$ are four $(n+1)\times (n+1)$-matrices with $e_{ij}=\langle \overline{f}_i, \overline{f}_j \rangle , 
\ b_{ij}=\langle \overline{f}_i, f_j \rangle, \ c_{ij}=\langle f_i,\overline{f}_j\rangle, \ d_{ij}=\langle f_i,f_j\rangle$ for $i,j=0, 1,...,n$. 
Then from \eqref{eq:5.1}, $E=D=I$ holds, i.e., they are both $(n+1)\times (n+1)$ identity matrices. We also have $C=\overline{B}$. Then 
$$R=\left(\begin{array}{cccccccccc}
I & B \\
\overline{B} & I \end{array}\right).$$
Of importance is the rank of $R$, which can be interpreted as the exact value of $N$.

\emph{Claim. $N=n+1$ or $N=2(n+1)$.}
\begin{proof}
	$$N=rank(I-\overline{B}B) + n+1$$ can be obtained 
from relation $R \left(\begin{array}{cccccccccc}
I & -B \\
0 & I \end{array}\right)=\left(\begin{array}{cccccccccc}
I & 0 \\
\overline{B} & I-\overline{B}B \end{array}\right)$. Since $\underline{f}_0$ is a Clifford solution, then by relations $a_i+a_{i+m+1}=0,  \forall i=0,1,...,m$ and $r_0=r_1=\cdots=r_n=\frac{1}{n+1}$,
\eqref{eq:4.20} \eqref{eq:4.21} and \eqref{eq:4.22} can be rewritten as 
\begin{equation} 
\langle f_i, \overline{f}_j \rangle=\frac{[(-1)^i+(-1)^j]}{n+1}(\sum_{k=0}^{m}a_k^{i+j}w_{k,m+1+k}), 
\label{eq:5.3}\end{equation}
\begin{equation}
\langle f_0, \overline{f}_j \rangle=\frac{[1+(-1)^j]}{n+1}(\sum_{k=0}^{m}a_k^{j}w_{k,m+1+k}),
\quad
\sum_{k=0}^{m}w_{k,m+1+k}=0.
\label{eq:5.4}\end{equation}
Then we get, for any $i,j$,  
$b_{ij}$ is a constant. Using this, we have following four facts.

(a) $b_{00}=b_{0,2m+2}=0$,

(b) $b_{ij}=0$ if $i+j$ is odd,

(c)  $b_{i+1,j}+b_{i,j+1}=0$,

(d) $b_{in}=-b_{0,i-1}$ if $i$ is odd.

Using (a)(b)(c)(d), it is not difficult  to check that $I-\overline{B}B$ is a circulant matrix. Let $g_{1k}$ denotes the element in the first row and $k$-th column, then 
\begin{align*}
g_{1k}=\begin{cases}
0,  \quad k=2,4,...,\\
1-(|b_{02}|^2+|b_{04}|^2+\cdots+|b_{0,2m}|^2), \quad k=1,\\
-(\overline{b}_{02}b_{0,k+1}+\overline{b}_{04}b_{0,k+3}+\cdots+\overline{b}_{0,2m-2}b_{0,k+2m-3}+\overline{b}_{0,2m}b_{0,k+2m-1}), \quad k=3,5,...,n. 
\end{cases}
\end{align*}
Then for circulant matrix,
$$\det(I-\overline{B}B)=F(a_0)F(a_1)...F(a_n),$$
where $$F(x)=g_{11}+g_{13}x^2+g_{15}x^4+\cdots+g_{1n}x^{n-1}.$$

If $\det(I-\overline{B}B)\neq 0$, then we have $rank(I-\overline{B}B)=n+1$ and thus derive $N=2(n+1)$. Otherwise if  $\det(I-\overline{B}B) = 0$. On one hand, there exists some $k\in {0,1,...,n}$ such that $F(a_k)=0$.  Therefore
\begin{align*}
g_{11}&:=-g_{13}a_k^2-g_{15}a_k^4-\cdots-g_{1n}a_k^{n-1}
\\&\leq|g_{13}|+|g_{15}|+\cdots+|g_{1n}|
\\&\leq \sum_{i,j=2,i\neq j}^{2m}|b_{0i}||b_{0j}|,
\end{align*}
which implies
\begin{equation}(|b_{02}|+|b_{04}|+\cdots+|b_{0,2m}|)^2 \geq 1.\label{eq:5.5}\end{equation}
On the other hand, $f_0, f_1,..., f_n$ form a unitary set in 
$\mathbb{C}^{n+1}$, which can be regarded as a fixed $(n+1)$-plane in $\mathbb{C}^{N}$. In other words, $e_{n+2},..., e_N$ can be choosen such that $f_0, f_1,..., f_n, e_{n+2},..., e_N$ form a unit orthogonal frame for $\mathbb{C}^{N}$. So we have
$$\overline{f}_0=b_{02}f_2+b_{04}f_4+\cdots+b_{0,2m}f_{2m}+\langle \overline{f}_0, e_{n+2} \rangle e_{n+2}+\cdots+\langle\overline{f}_0, e_{N} \rangle e_{N}.$$  Together with  relation $|f_i|=1, \forall i\in \mathbb{Z}$, we obtain
$$(|b_{02}|+|b_{04}|+\cdots+|b_{0,2m}|)^2 \leq 1.$$
Comparing it with \eqref{eq:5.5},
$$\overline{f}_0=b_{02}f_2+b_{04}f_4+\cdots+b_{0,2m}f_{2m}.$$
It concludes that $N=n+1$. 
\end{proof} 
Following our claim, when $N=2(n+1)$, $\overline{f}_0, \overline{f}_1,..., \overline{f}_n, f_0, f_1,..., f_n$ is a unit orthogonal frame for $\mathbb{C}^{N}$. In particular, $$\langle f_0, \overline{f}_j \rangle=0, \forall j\in \mathbb{Z}.$$ 
Then similar discussions about the first equation in \eqref{eq:5.4} give 
$$w_{k,m+1+k}=0, \quad k=0,...,m.$$
In this case, $W$ is of the same form as \eqref{eq:4.11} and $f_0$ can be expressed by \eqref{eq:5.2}. 

When $N=n+1$. $f_0$ can be shown as \eqref{eq:4.50} under relation $r_0=\cdots=r_n=\frac{1}{n+1}$,  then
\begin{equation}f_0=\frac{1}{\sqrt{2N}}\left(\begin{array}{ccccccc}
\sqrt{w_{0,m+1}}(e^{z-\overline{z}}+e^{\overline{z}-z})
\\
\sqrt{w_{1,m+2}}(e^{a_1z-\overline{a}_1\overline{z}}+e^{\overline{a}_1\overline{z}-a_1z})
\\
\vdots  
\\
\sqrt{w_{mn}}(e^{a_mz-\overline{a}_m\overline{z}}+e^{\overline{a}_m\overline{z}-a_mz})
\\
\sqrt{-w_{mn}}(e^{a_mz-\overline{a}_m\overline{z}}-e^{\overline{a}_m\overline{z}-a_mz})
\\
\vdots
\\
\sqrt{-w_{1,m+2}}(e^{a_1z-\overline{a}_1\overline{z}}-e^{\overline{a}_1\overline{z}-a_1z})
\\
\sqrt{-w_{0,m+1}}(e^{z-\overline{z}}-e^{\overline{z}-z})
\end{array}\right),\label{eq:5.6}\end{equation}
where $w_{0,m+1}+w_{1,m+2}+\cdots+w_{mn}=0$ and $a_i=e^{\frac{2\pi \sqrt{-1}i}{n+1}}, i=0,1,2,...,n$ are the $n+1$-roots of unity.

 {\bfseries Case III: $\exists i, \forall j  \ s.t. a_i+a_j \neq 0$ and $\exists k, l \ s.t. a_k+a_l=0$.}  
 
 Here we distinguish two subcases for consideration: $n$ is odd, or $n$ is even.
 
 We first briefly discuss the case when $n$ is odd. By a similar discussion as Case II above, we get $N=n+1$ or $N=2(n+1)$. The first one is impossible since there exists some $i_{\beta}$ s.t. $w_{i_{\beta}k}=0$ for $k=0,1,...,n$.
 So $N=2(n+1)$,   $W$ is of the same form as \eqref{eq:4.11} and $f_0$ can be expressed by \eqref{eq:5.2}. 
 
 Next we  consider the case when $n$ is even. Write $n=2m$ and  suppose $a_{i_\alpha}+a_{j_\alpha}=0, \ 0\leq i_\alpha<j_\alpha\leq n$ for $\alpha=1,...,s$ and for $\beta=2s+1,...,n+1, \forall j, a_{i_\beta}+a_{j}\neq 0, 0\leq i_\beta \leq n$. Then except for $w_{i_1j_1}, w_{i_2j_2},..., w_{i_sj_s}$ we have 
 $w_{ij}=0, i, j=0,1,2,...,n$.
 Under these assumptions, we shall show
 \begin{equation}w_{i_1j_1}=w_{i_2j_2}=\cdots=w_{i_sj_s}=0,\label{eq:5.7}\end{equation}
 which means  $W$ is given as \eqref{eq:4.11}. To do this, we firstly need to consider $\langle f_0, \overline{f}_{k} \rangle$.
In fact for any $k\in \mathbb{Z}$
\begin{equation}
\langle f_0, \overline{f}_{k} \rangle=\frac{1+(-1)^k}{n+1}(w_{i_1j_1}a^k_{i_1}+w_{i_2j_2}a^k_{i_2}+\cdots+w_{i_sj_s}a^k_{i_s}).
\label{eq:5.8}\end{equation}
Combining with $f_{n+k}=f_{k-1}$, following facts can be obtained
\\(e) $\langle f_0, \overline{f}_{0} \rangle=\langle f_0, \overline{f}_{n+1} \rangle=0$,
\\(f) $\langle f_0, \overline{f}_{k} \rangle$ is a constant for any $k\in \mathbb{Z}$,
\\(g) $\langle f_0, \overline{f}_{k} \rangle=0$ if $k=1,3,5,...,2m-1$.
\\(h) $\langle f_0, \overline{f}_{n+1+k} \rangle=\langle f_0, \overline{f}_{k} \rangle=0$ for $k=0,1,...,n$.

From these properties, for $k=0,1,...,m$ we get the following useful relation
$$\langle f_0, \overline{f}_{n+2k} \rangle=0,$$
which implies
$$w_{i_1j_1}a^{2k-1}_{i_1}+w_{i_2j_2}a^{2k-1}_{i_2}+\cdots+w_{i_sj_s}a^{2k-1}_{i_s}=0$$ by using of $a^{n+1}_{i_1}=a^{n+1}_{i_2}=\cdots=a^{n+1}_{i_s}=1$. Rewrite it in the following form 
$$\left(\begin{array}{ccccccc}
a_{i_1} & a_{i_2} & \cdots & a_{i_s} \\
a^3_{i_1} & a^3_{i_2} & \cdots & a^3_{i_s}  \\
\vdots & \vdots &  &\vdots  \\
a^{2s-1}_{i_1} & a^{2s-1}_{i_2} & \cdots & a^{2s-1}_{i_s}
\end{array}\right)
\left(\begin{array}{ccccccc}
w_{i_1j_1} \\
w_{i_2j_2} \\
\vdots  \\
w_{i_sj_s}
\end{array}\right)
=\left(\begin{array}{ccccccc}
0 \\
0 \\
\vdots  \\
0
\end{array}\right).$$
Then \eqref{eq:5.7} can be obtained since the determinant of 
coefficient matrix of the above equation is
$$a_{i_1}a_{i_2}\cdots a_{i_s}\prod_{i_q < i_p}(a^2_{i_p}-a^2_{i_q})\neq 0.$$
Then we still have $N=2(n+1)$,   $W$ is of the same form as \eqref{eq:4.11} and $f_0$ can be expressed by \eqref{eq:5.2}.
 
 In summary,  we obtain a classification of Clifford solutions as follows:
 \begin{theo}
 	Let $f: \mathbb{C} \rightarrow Q_{N-2}$ be a linearly full Clifford solution. Then $f$ is congruent to \eqref{eq:5.2} or \eqref{eq:5.6}.
 \end{theo}

\end{document}